\documentclass[reqno]{amsart}
\usepackage[utf8]{inputenc}
\usepackage[a4paper]{geometry}

\usepackage[english]{babel}
\usepackage[document]{ragged2e}
\usepackage{amsfonts}
\usepackage{amsthm}
\usepackage{amssymb, amsmath}
\usepackage{hyperref}

\usepackage{xcolor}
\usepackage{cleveref}
\usepackage{enumitem}

\theoremstyle{plain}
\newtheorem{theo}{Theorem}[section]

\newtheorem{prop}[theo]{Proposition}
\newtheorem{lemma}[theo]{Lemma}

\newtheorem{coro}[theo]{Corollary}

\theoremstyle{definition}
\newtheorem{rem}[theo]{Remark}
\newtheorem{example}[theo]{Example}

\newcommand{\D}{\mathbb{D}}

\title[Omitted Values for some subclasses of univalent mappings]{Omitted Values for some subclasses of univalent mappings}

\date{\today}

\author[H. Arbeláez]{Hugo Arbeláez}
\address{Hugo Arbeláez, Escuela de Matemáticas, Universidad Nacional de Colombia, Medellín, Colombia}
\email{hjarbela@unal.edu.co}

\author[R. Hernández]{Rodrigo Hernández}
\address{Rodrigo Hernández, Facultad de Ingeniería y Ciencias, Universidad Adolfo Ibáñez, Av. Padre Hurtado 750, Viña del Mar, Chile}
\email{rodrigo.hernandez@uai.cl}

\author[W. Sierra]{Willy Sierra}
\address{Willy Sierra, Departamento de Matemáticas, Universidad del Cauca, Popayán, Colombia}
\email{wsierra@unicauca.edu.co}

\thanks{The first author was supported by the Universidad Nacional de Colombia code Hermes 66371. The fourth author thanks the Universidad del Cauca for providing time for this work through research project VRI ID 6671.}

\begin{document}

\maketitle

\justifying

\begin{abstract}
We study the range of $\operatorname{Re}\{a_2 f(z)\}$ for normalized analytic functions $f$ in the unit disk belonging to several classes of conformal mappings. As our main contribution, we introduce the class $CC_\alpha$ of \emph{completely convex mappings of order $\alpha$}, defined by a uniform two-point starlikeness condition, and we estimate the range of $\operatorname{Re}\{a_2f(z)\}$ in terms of $\alpha$, for all $f\in CC_\alpha$ and $z\in \D$, generalizing the classical result of Fournier--Ma--Ruscheweyh, which is recovered for $\alpha=0$. We also determine omitted value sets for convex functions of order $\alpha$, spherically convex mappings, uniformly starlike functions, and Nehari classes $\mathcal{N}_t$. The proofs rely primarily on the Schwarz--Pick lemma applied to auxiliary functions constructed from the two-point kernel $zf'(z)/(f(z)-f(x))$.
\end{abstract}
\vspace{0.3cm}


\section{Introduction and main results}
 
The main theme of the paper concerns the study of values that are either assumed or omitted by functions belonging to certain classes of univalent mappings. This problem traces back to Koebe's one-quarter theorem \cite{GK2003}, which establishes that the image of the unit disk $\mathbb{D}=\left\lbrace z: |z|<1 \right\rbrace $ under any univalent function $f:\mathbb{D}\to\mathbb{C},$ normalized by $f(z) = z +a_2z^2+\cdots,$ contains a disk of radius one-quarter centered at the origin. Building on this classical result, the investigation of omitted values has become a central theme in geometric function theory. Two important contributions to this topic were made by Goodman \cite{G1949}, who established a sharp upper bound for the radius of the circle with a given center such that all values inside it are omitted by $f,$ and Jenkins \cite{J1953}, who studied the set of values on the circle $|w|=r,$ $1/4<r<1,$ that are omitted by normalized univalent functions. Decades later,  Chuaqui and Osgood \cite{CO1993}, in their studies on univalent functions and the Schwarzian derivative, showed that every normalized function in the Nehari class $\mathcal{N} = \{f : |S_f(z)| \leq 2(1-|z|^2)^{-2}\}$ does not take the value $-1/a_2$ in $\D$. This result was subsequently improved by Fournier \textit{et al.} \cite{FMR1997}, who proved that $\sigma(\mathcal{N})=\mathbb{C} \setminus\{-1\},$ where, for a class $\mathcal{W}$ of normalized analytic functions on $\mathbb{D}$, $\sigma(\mathcal{W})$ denotes the set of values taken by $a_2 f(z)$ as $f$ ranges over $\mathcal{W}$ and $z$ over $\mathbb{D}.$ In that same paper, the authors also established that for the subclass $\mathcal{C}$ of $\mathcal{N}$ consisting of convex functions, $\sigma(\mathcal{C})=\{w:\operatorname{Re}\{w\}>-1/2\}$. More precisely, they established the sharp inequality \[\operatorname{Re}\{a_2f(z)\}\geq-\frac{1}{2}+\frac{1}{2}(1-|z|^2) \left|\frac{f(z)}{z}\right|^2\geq -\frac{1}{2},\]
for all $f\in\mathcal{C}$ and $z\in\mathbb{D}.$ We also comment that, at the opposite extreme, for the class $S^\ast$ of normalized starlike mappings, $\sigma(S^\ast)=\mathbb{C}.$ In this paper we continue this line of research by establishing several results, among which we highlight estimates for $\operatorname{Re}\{a_2f(z)\}$ and $\operatorname{Re}\{f(z)/z\},$ where $f$ belongs to a number of geometrically natural classes. The main results focus on a new class of functions introduced from the notion of convexity of order $\alpha.$ The key observation is that a conformal map $f$ is convex if and only if $f(\D)$ is starlike with respect to each of its points, a global two-point condition given by
\[\operatorname{Re}\left\{\frac{2zf'(z)}{f(z)-f(x)}-\frac{z+x}{z-x}\right\}\geq 0, \qquad z, x\in\D.\]
Requiring $f(\D)$ to be starlike of order $\alpha$ with respect to all its points leads naturally to the class of \emph{completely convex mappings of order $\alpha$}, denoted $CC_\alpha,$ consisting of functions $f(z)=z+a_2z^2+\cdots$ that satisfy
\[\operatorname{Re}\left\{\frac{2zf'(z)}{f(z)-f(x)}-\frac{z+x}{z-x}\right\}\geq \alpha, \qquad z, x\in\D.\]
According to Theorem\,1 in \cite{S1970}, we have $CC_0 = \mathcal{C}.$ Letting $x\to z$ gives $CC_\alpha \subseteq C_\alpha$, and setting $x=0$ yields $CC_\alpha \subset S^*_{(1+\alpha)/2}$, where $C_\alpha$ and $S^\ast_\alpha$ denote the classical classes of convex and starlike functions of order $\alpha,$ introduced by Robertson \cite{Rob1936}:
\[\operatorname{Re}\left\{1+z\frac{f''}{f'}(z)\right\}\geq \alpha \qquad \text{and} \qquad \operatorname{Re}\left\{z\frac{f'}{f}(z)\right\}\geq \alpha,\]
respectively. I. S. Jack \cite{J1971} proved that $C_\alpha\subset S^\ast_{\beta}$ with $\beta=\frac{2\alpha-1+\sqrt{(2\alpha-1)^2+8}}{4}$; in particular, $\mathcal{C} = C_0 \subset S^\ast_{1/2}$. The class $CC_\alpha$ interpolates between $\mathcal{C}$ and the Möbius transformations, and, as we show, it provides sharper omitted value sets than $C_\alpha$ precisely because of its uniform geometric character.

The paper is organized as follows. In Section\,\ref{subsec:CC} we study the class $CC_\alpha$ and obtain estimates for $\operatorname{Re}\left\lbrace f(z)/z\right\rbrace$ and $\operatorname{Re}\left\lbrace a_2f(z) \right\rbrace,$ with $f\in CC_\alpha$ and $z\in\mathbb{D}.$ Similar results are derived for convex mappings of order $\alpha$ in Section\,\ref{convex order alpha}, for spherically convex mappings in Section\,\ref{subsec:sconv}, for uniformly starlike functions in Section\,\ref{subsec:UCV}, and finally, in Section\,\ref{subsec:Nehari} we address the Nehari class.

\section{Omitted values}
\label{sec:omitted}

\subsection{Completely convex mappings of order $\alpha$}
\label{subsec:CC}

For $\alpha\in [0,1)$, the class $CC_{\alpha}$ consists of analytic functions $f(z)= z+a_2z^2+\cdots$ on $\D$ satisfying
\begin{equation} \label{conv1}
\operatorname{Re}\,\left\{\frac{2zf'(z)}{f(z)-f(x)}-\frac{z+x}{z-x}\right\}\geq \alpha,\qquad z, x\in\D.
\end{equation}
As noted in the introduction, $CC_0=\mathcal{C}$, $CC_\alpha \subseteq C_\alpha$, and $CC_\alpha \subset S^*_{(1+\alpha)/2}$.

\begin{rem}
The condition \eqref{conv1} has a clean geometric meaning: $f \in CC_\alpha$ if and only if, for each $x \in \D$, the domain $f(\D)$ is starlike of order $(\alpha+1)/2$ with respect to the point $f(x)$. This is a \emph{global} two-point condition, in contrast to the \emph{infinitesimal} one-point condition that characterizes $C_\alpha$ (where convexity of order $\alpha$ is only tested as $x\to z$). The uniformity in $x$ is precisely what allows $CC_\alpha$ to yield sharper omitted value estimates.
\end{rem}

\begin{example}
 Let $T(z)=\dfrac{az+b}{cz+d}$ with $ad-bc=1$ and $\zeta:=-d/c\notin \overline{\D}$. A direct calculation gives
 \begin{equation*}
\frac{2zT'(z)}{T(z)-T(x)}-\frac{z+x}{z-x}=\frac{d-cz}{d+cz},    
\end{equation*}
and therefore
\begin{equation*}
  \operatorname{Re}\,\left\{\frac{2zT'(z)}{T(z)-T(x)}-\frac{z+x}{z-x}\right\} = \frac{|d|^2-|cz|^2}{|d+cz|^2}\geq \frac{|d|-|c|}{|d|+|c|}=\frac{|\zeta|-1}{|\zeta|+1}.
\end{equation*}
Choosing $|\zeta|=\dfrac{1+\alpha}{1-\alpha}$ gives $T\in CC_{\alpha}$, showing that the Möbius transformations with pole outside $\overline{\D}$ belong to $CC_\alpha$ for a suitable $\alpha$ depending on the distance of the pole from the disk. This reveals something noteworthy: for a fixed $\alpha$, the class $CC_\alpha$ does not contain all Möbius transformations. That is, this family is remarkably restrictive, which explains why it appears so naturally across the classical literature on conformal functions.
\end{example}

\begin{rem}\label{rem:subordination}
Condition \eqref{conv1} is equivalent to the representation
 \begin{equation}\label{Obs2}
  \frac{2zf'(z)}{f(z)-f(x)}-\frac{z+x}{z-x}=\frac{1+(1-2\alpha)\omega(x,z)}{1-\omega(x,z)},
 \end{equation}
 where $\omega:\D \times \D \to \D$ with $\omega(x,0)=0$. Writing $\omega(x,z)=z\delta(x,z)$ then yields
 \begin{equation*}
  \frac{f'(z)}{f(z)-f(x)}-\frac{1}{z-x}=\frac{(1-\alpha)\delta(x,z)}{1-z\delta(x,z)},   
 \end{equation*}
 for all $x,z\in \D$. In the language of differential subordinations \cite{MM2000}, the equation \eqref{Obs2} states that for each fixed $x$, the function $z\mapsto \frac{2zf'(z)}{f(z)-f(x)}-\frac{z+x}{z-x}$ is subordinate to the right half-plane mapping $\frac{1+(1-2\alpha)z}{1-z}$, uniformly in $x\in \D$.
\end{rem}
\begin{example}
Let $a\in \mathbb{R}^+$ and $f(z)= z+az^n$, $z\in\D$, $n\in \mathbb{N}$.
We observe that
\begin{align*}
 \frac{f'(z)}{f(z)-f(x)}-\frac{1}{z-x}
&= \frac{1 + a n z^{\,n-1}}{z - x + a(z^n - x^n)}- \frac{1}{z-x}\\
&= \frac{a n z^{\,n-1} - a\big(z^{n-1} + z^{n-2}x + z^{n-3}x^2 + \cdots + x^{n-1}\big)}
{(z-x)\big(1 + a(z^{n-1} + z^{n-2}x + \cdots + x^{n-1})\big)}\\
&=a \frac{z^{n-2} + z^{n-3}x + z^{n-4}x^2 + \cdots + (z^{n-2} + z^{n-3} + \cdots + x^{n-2}) (x+z)}
{1 + a\big(z^{n-1} + z^{n-2}x + \cdots + z x^{n-2} + x^{n-1}\big)}\\
&= \frac{aP(x,z)}{1 + aQ(x,z)}.
\end{align*}
Thus, from the previous observation, $f\in CC_\alpha$ if
\[
\frac{(1-\alpha)\delta(x,z)}{1-z\delta(x,z)}= 
\frac{aP(x,z)}{1+aQ(x,z)},
\]
and therefore,
\[
\delta(x,z)=
\frac{aP(x,z)}{1-\alpha+a\left[(1-\alpha)Q(x,z)+zP(x,z)\right]\,}.
\]
It follows that
\begin{align}\label{exam2}
|\delta(x,z)|&\leq \frac{a|P(x,z)|}{1-\alpha-a\big|(1-\alpha)Q(x,z)+zP(x,z)\big|} \nonumber\\
&\leq \frac{\tfrac{1}{2} an(n-1)}{1-\alpha-a\left[(1-\alpha)n+\tfrac{n(n-1)}{2}\right]}.
\end{align}
The right side in (\ref{exam2}) is less than or equal to $1$ if $a\leq\dfrac{1-\alpha}{n(n-\alpha)}$.
\end{example}

\begin{lemma}\label{start}
If $f \in CC_\alpha$, then
\[\frac{1}{(1+|z|)^{1-\alpha}}\leq \left|\frac{f(z)}{z}\right|\leq \frac{1}{(1-|z|)^{1-\alpha}}\,,\]
for all $z\in\mathbb{D}$.
\end{lemma} 

\begin{proof}
Since $f \in \mathcal{S}^*_{(1+\alpha)/2}$, it follows that
\[
\frac{|z|}{(1+|z|)^{2(1-\frac{1+\alpha}{2})}}\leq|f(z)|\leq \frac{|z|}{(1-|z|)^{2(1-\frac{1+\alpha}{2})}},
\qquad z \in \mathbb{D}.\] See \cite[Pag. 57]{GK2003}
\end{proof}

An important consequence of the following result is to generalize the known result to convex functions $\operatorname{Re}\{f(z)/z\}>1/2$.

\begin{theo}\label{CCC}
Let $\alpha\in [0,1)$ and $f\in CC_{\alpha},$ then
 \begin{equation*}
 \operatorname{Re}\left\{\dfrac{f(z)}{z}\right\}\geq m(|z|), \qquad z\in\mathbb{D},
 \end{equation*}
 where $m(t)=\max\left\lbrace (1/2)^{1-\alpha},\, p(t) \right\rbrace$ and $p(z)=\displaystyle \frac{1}{2}\left[1+\frac{1-(1-\alpha)^2z^2}{(1+z)^{2(1-\alpha)}}\right],$ $z\in\mathbb{D}.$
\end{theo}

\begin{proof}
Before we proceed with the proof, let us remark that when $\alpha=0,$ $p(|z|)\geq (1/2)^{1-\alpha},$ and when $\alpha=1,$ $p(|z|)= (1/2)^{1-\alpha},$ for all $z\in\mathbb{D}.$ In the case $0<\alpha<1,$ we have $p(|z|)\geq (1/2)^{1-\alpha},$ for $0\leq |z|\leq a,$ where $a$ is the root of the equation $p(|z|)=(1/2)^{1-\alpha}.$ Consequently, to prove the theorem we will proceed to show that conditions
\[ i)\;\; \operatorname{Re}\left\{\dfrac{f(z)}{z}\right\}\geq \left(\dfrac{1}{2}\right)^{1-\alpha}\qquad\text{and}\qquad ii)\;\; \operatorname{Re}\left\{\dfrac{f(z)}{z}\right\}\geq p(|z|),\]
are satisfied for all $z\in\mathbb{D}.$
  
We begin by proving condition i). By (\ref{conv1}), with $x=0$, we have that
\[
\operatorname{Re}\left\{\frac{z f'}{f}(z)\right\} \geq \frac{1+\alpha}{2},
\qquad z \in \mathbb{D},
\]
which is equivalent to
\[
\operatorname{Re}\left\{2\frac{\,\frac{z f'}{f}(z)-\alpha}{1-\alpha}-1\right\}\geq 0.
\]
We define
\begin{equation}\label{start1}
 \frac{zg'}{g}(z)=\frac{\frac{z f'}{f}(z)-\alpha}{1-\alpha}.   
\end{equation}
Thus, $g \in \mathcal{S}^*_\frac{1}{2}$, which implies $\operatorname{Re}\left\{\dfrac{g(z)}{z}\right\} > \dfrac{1}{2}$ (see \cite{R1982}, p. 49). On the other hand, it follows from (\ref{start1}) that
\[
\frac{f(z)}{z}=\left(\frac{g(z)}{z}\right)^{1-\alpha},
\]
whence
\begin{equation}\label{kk}
\operatorname{Re}\left\{\frac{f(z)}{z}\right\}
= \left|\frac{g(z)}{z}\right|^{1-\alpha}
\cos\big((1-\alpha)\theta\big)
> \left(\frac{1}{2}\right)^{1-\alpha}
\frac{\cos\big((1-\alpha)\theta\big)}{\cos^{1-\alpha}\theta}\geq \left(\frac{1}{2}\right)^{1-\alpha},  
\end{equation}
for all $z\in \D$, where $\theta=\arg \Big\{\dfrac{g(z)}{z}\Big\}$. The last inequality in \eqref{kk} it follows from the concavity of $\log(\cos(x))$ on $(-\pi/2,\pi/2)$ applied to $\cos((1-\alpha)\theta) = \cos((1-\alpha)\theta + \alpha\cdot 0) \geq \cos^{1-\alpha}\theta\cdot\cos^\alpha(0) = \cos^{1-\alpha}\theta$.

Next, we proceed to prove condition ii). For all $x, z\in \D$, we define 
\begin{equation*}
G(z,x)= \left\{ \begin{array}{lrl} \dfrac{2zf'(z)}{f(z)-f(x)}-\dfrac{z+x}{z-x}; &  z \neq x, \\ \\ \dfrac{zf''}{f'}(z)+1; & z=x. \end{array} \right. 
\end{equation*}  
For fixed $x$, let $G(z)=G(z,x)$ which is analytic. We observe that $\operatorname{Re}\,G(z)\geq \alpha$ and $G(0)=1$. Let $H:\D \to \D$ be defined by 
\begin{equation*}
H(z)=\frac{G(z)-1}{G(z)+1-2 \alpha}.
\end{equation*}
A straightforward computation shows that
$$H'(0)=\frac{G'(0)}{2(1-\alpha)}=\frac{a_1(G)}{2(1-\alpha)}$$
and
$$\frac{H''(0)}{2}=\frac{(1-\alpha)^2G''(0)-(1-\alpha)(G'(0))^2}{4(1-\alpha)^3}=\frac{a_2(G)}{2(1-\alpha)}-\frac{a_1^2(G)}{4(1-\alpha)^2}.$$
By the Schwarz--Pick lemma
\begin{equation}\label{conv2}
  |2(1-\alpha)a_2(G)- a_1^2(G)|\leq 4(1-\alpha)^2-|a_1^2(G)|.
\end{equation}
Since $G(z)=2\left[\dfrac{zf'(z)}{f(z)-f(x)}-\dfrac{x}{z-x}\right]-1$, we can see, as in \cite{FMR1997}, that 
\begin{equation*}
a_1(G)=2\left(\frac{1}{x}-\frac{1}{f(x)}\right) \quad \text{and} \quad a_2(G)=2\left(\frac{1}{x^2}-\frac{1}{f(x)}\left(2a_2+\frac{1}{f(x)}\right)\right).
\end{equation*}
Replacing in (\ref{conv2}) and simplifying, we obtain
\begin{equation}\label{conv3}
 \left|\frac{2f(x)}{x}+\alpha -2(1-\alpha)a_2f(x)-\alpha\left(\frac{f(x)}{x}\right)^2-2\right|\leq (1-\alpha)^2|f(x)|^2-\left|\frac{f(x)}{x}\right|^2+2\operatorname{Re}\left\{\frac{f(x)}{x}\right\}-1,
\end{equation}
from this and Lemma \ref{start}, it follows
\begin{equation*}
 2\operatorname{Re}\,\left\{\frac{f(x)}{x}\right\}\geq 1+\left|\frac{f(x)}{x}\right|^2(1-(1-\alpha)^2|x|^2)\geq 1+\dfrac{1}{(1+|x|)^{2(1-\alpha)}}(1-(1-\alpha)^2|x|^2). 
\end{equation*}
\end{proof}

\begin{prop} \label{F_a1}
  Let $f\in CC_{\alpha}$, $\alpha\in [0,1)$, then for all $a\in \D$
  \begin{equation*}
      F_a(z):=\frac{az}{f(a)}\frac{f(z)-f(a)}{z-a}
  \end{equation*}
is starlike of order $(\alpha +1)/2$. Also, $\operatorname{Re}\left\{\dfrac{F_a(z)}{z}\right\}  > \left(\dfrac{1}{2}\right)^{1-\alpha}$ for all $z\in \D$.
\end{prop}
\begin{proof}
    By direct calculation
\begin{equation*}
    \frac{zF'_a}{F_a}(z)=\frac{zf'(z)}{f(z)-f(a)}-\frac{a}{z-a}.
\end{equation*}
So, by (\ref{conv1})
\begin{equation*}
    \alpha\leq \operatorname{Re}\,\left\{\frac{2zf'(z)}{f(z)-f(a)}-\frac{2a}{z-a}-1\right\}=2\operatorname{Re}\left\{\frac{zF'_a}{F_a}(z)\right\}-1,
\end{equation*}
and it follows that $\operatorname{Re}\left\{\dfrac{zF'_a}{F_a}(z)\right\}\geq (\alpha +1)/2$ for all $z\in \D$. The last statement follows as in the proof of Theorem \ref{CCC} item i).
\end{proof}

\begin{coro}
  Let $f\in CC_{\alpha}$, $\alpha\in [0,1)$, then for all $a\in\D$, the function $h(z)=z^{-1}F^2_a(z)$ is starlike of order $\alpha$.
 \end{coro}
   
  \begin{proof}
  It is enough to observe that
 \begin{equation*}
\operatorname{Re}\left\{\dfrac{zh'}{h}(z)\right\}=2\operatorname{Re}\left\{\dfrac{zF'_a}{F_a}(z)\right\}-1
\end{equation*}
 for all $z\in \D$.
\end{proof}
Note that when $a=0$, $F_0(z)=f(z)$, thus $z^{-1}f^2(z)\in S^*_\alpha$ for every $f\in CC_\alpha$.

\begin{coro}
 Let $f\in CC_{\alpha}$, $\alpha\in [0,1)$, and $|z|=r<1$. Then
 \begin{equation}\label{conv4}
     \frac{1}{(1+r)^{1-\alpha}}\leq \left|\frac{zf'(z)}{f(z)}\right| \leq \frac{1}{(1-r)^{1-\alpha}}.
 \end{equation}
\end{coro}
\begin{proof}
Fixed $z\in \D$. Since $F_a\in S^*_{(\alpha+1)/2}$ for all $a\in \D$, we have by Theorem 2.3.7 in \cite{GK2003} that
\begin{equation*}
    \frac{r}{(1+r)^{1-\alpha}}\leq \left|\frac{az}{f(a)}\frac{f(z)-f(a)}{z-a}\right| \leq \frac{r}{(1-r)^{1-\alpha}}. 
\end{equation*}
 Now, (\ref{conv4}) it follows when $a$ tends to $z$.
\end{proof}

\begin{coro}
 Let $f\in CC_{\alpha}$, $\alpha\in [0,1)$, then
 \begin{equation} \label{Coro3}
\left|\frac{f(z)}{z}-\frac{f(a)}{a}\right|\leq |f(z)-f(a)|+(1-2^\alpha)\frac{|z-a|}{(1-|a|)^{1-\alpha}}, \quad a,z \in \D.
 \end{equation}
\end{coro}

\begin{proof}
Fixed $a\in \D$ and let $h(z)=\frac{F_a(z)}{z}$. By Proposition \ref{F_a1} we have $\operatorname{Re}{h(z)}\geq \frac{1}{2^{1-\alpha}}$, then
\[
\operatorname{Re}\left\{
\frac{2^{1-\alpha}h(z)-1}{2^{1-\alpha}-1}
\right\}\ge 0.
\]
Hence, there exists a function $w:\mathbb{D} \to \mathbb{D}$, with $w(0)=0$, such that 
\[
\frac{2^{1-\alpha}h(z)-1}{2^{1-\alpha}-1}
=\frac{1+w(z)}{1-w(z)}.
\]
Therefore,
\[
H(z):=\frac{h(z)-1}{2-2^\alpha}=\frac{w(z)}{1-w(z)},
\]
which implies $w(z)=\frac{H(z)}{1+H(z)}$. By Schwarz's Lemma,
\[
\left|\frac{h(z)-1}{h(z)+1-2^\alpha} \right|\leq |z|
\]
equivalently,
\begin{equation}\label{C.10}
 \left|\frac{\dfrac{a}{f(a)}\dfrac{f(z)-f(a)}{z-a}-1}{\dfrac{a}{f(a)}\dfrac{f(z)-f(a)}{z-a}+1-2^\alpha}\right|\leq |z|.   
\end{equation}
It follows, by \eqref{C.10} and Lemma \ref{start}, that
\begin{align*}
  \left|\frac{f(z)}{z}-\frac{f(a)}{a}\right|&
\leq\left|f(z)-f(a)+(1-2^\alpha)\frac{f(a)}{a}(z-a)\right|\\& \leq  \left|f(z)-f(a)\right|+(1-2^\alpha)\frac{|z-a|}{(1-|a|)^{1-\alpha}}.
\end{align*}
\end{proof}
The following is the main result of this section, which generalizes Theorem~1 of \cite{FMR1997} to the class $CC_\alpha$. The argument follows the same strategy as in \cite{FMR1997}, with the key difference that the lower bound for $|f(z)/z|$ given by Lemma \ref{start} replaces the classical estimate $\operatorname{Re}\{f(z)/z\}>1/2$ for convex functions.

\begin{theo}\label{main}
Let $f(z)=z+a_2z^2+\cdots \in CC_\alpha$ and $0 \leq \alpha<1$. Then
\[
\operatorname{Re}\{ a_2 f(z)\}
\geq -1+\left(\frac{1}{2}\right)^{1-\alpha}+\dfrac{1-4\left(-1+\left(\frac{1}{2}\right)^{1-\alpha}\right)^2|z|^2}{{2(1+|z|)^{2(1-\alpha)}}},
\]
for all $z\in\mathbb{D}$.
\end{theo}

Observe that when $\alpha=0$ the right-hand side reduces to $$\operatorname{Re}\{a_2f(z)\}\geq -\frac12+\frac{1-|z|^2}{2(1+|z|)^2}=-\frac{|z|}{(1+|z|)}\geq -\frac12.$$ This inequality can be found in \cite{FMR1997}. 
\begin{proof}
By Proposition \ref{F_a1}, we have $F_a \in \mathcal{S}^*_{(\alpha+1)/2}$,
which implies that $F_a \in \mathcal{S}^*_{\left( 1/2 \right)^{1-\alpha}}$. Let
\[
G_a(z):=\frac{z F_a'(z)}{F_a(z)},
\]
thus $\operatorname{Re}\{ G_a(z) \} \ge \left( \frac{1}{2} \right)^{1-\alpha}$, for all $z \in \mathbb{D}$, and $G_a(0) = 1$. Now define
\[
H(z)=\frac{G_a(z) - \left( \frac{1}{2} \right)^{1-\alpha}}
{1-\left(\frac{1}{2}\right)^{1-\alpha}}.
\]
Then there exists a function $\omega:\mathbb{D} \to \mathbb{D}$, with $\omega(0)=0$, such that
\[
H(z)=\frac{1+\omega(z)}{1-\omega(z)}.
\]
Assume that
\[
\omega(z)=w_1 z + w_2 z^2 + \cdots = z (w_1+w_2 z+\cdots),
\]
by the Schwarz--Pick lemma we have $|w_2|\leq 1-|w_1|^2$, which implies
\begin{equation} \label{F_alpha1}
 |2H_2-H_1^2|\leq 4-|H_1|^2,
\end{equation}
where $H(z) = 1+H_1 z+H_2 z^2+\cdots.$ Moreover,
\[
H_1=\frac{G_1}{1-\left( \frac{1}{2} \right)^{1-\alpha}}
\qquad \text{and} \qquad
H_2= \frac{G_2}{1 - \left( \frac{1}{2} \right)^{1-\alpha}},
\]
here $G_a(z) = 1 + G_1 z + G_2 z^2 + \cdots.$ Substituting into (\ref{F_alpha1}), we obtain
\begin{equation}\label{F_alpha2}
|2\beta G_2-G_1^2|\leq 4\beta^2-|G_1|^2,
\end{equation}
with $\beta := 1 - \left( \frac{1}{2} \right)^{1-\alpha}.$ Using the idea of Fournier \emph{et al.} in \cite{FMR1997}, with $x$ in place of $a$,
\[
G_1=\frac{1}{x}-\frac{1}{f(x)} 
\qquad \text{and} \qquad
G_2=\frac{1}{x^2}-\frac{1}{f(x)}\left(2a_2+\frac{1}{f(x)}\right).
\]
After replacing these expressions in (\ref{F_alpha2}) and performing some algebraic manipulations, we obtain
\begin{align*}
 |x|^2\left(2\operatorname{Re}\left\{\frac{f(x)}{x}\right\}-1\right)-|f(x)|^2(1-4\beta^2 |x|^2)& \geq 2|x|^2\left|\frac{f(x)}{x}-\left(\beta+\frac{1}{2}+a_2f(x)\right)
\right|\\
&\geq 2|x|^2\left(\operatorname{Re}\left\{\frac{f(x)}{x}\right\}-\beta-\frac{1}{2}
- \operatorname{Re}\{a_2f(x)\}
\right),
\end{align*}
it follows that
\[
\operatorname{Re}\{a_2f(x)\}\geq -\beta +\frac{1}{2}\left|\frac{f(x)}{x}\right|^2(1-4\beta^2|x|^2).\] Using Lemma \ref{start} we have $$\operatorname{Re}\{a_2f(x)\}\geq -\beta +\dfrac{1-4\beta^2|x|^2}{2(1+|x|)^{2(1-\alpha)}}\,.$$Thus, the proof is completed.
\end{proof}

It is worth noting that the previous theorem recovers Theorem~1 in \cite{FMR1997} when $\alpha=0$, since $CC_0=\mathcal{C}$. Moreover, the minimum value of the right side in the inequality of Theorem \ref{main} is $$-\beta+\frac{1-4\beta^ 2}{2\cdot 4^{1-\alpha}}=-\left(1-\left(\frac 12\right)^{1-\alpha}\right)+\frac{1-4\left(1-\left(\frac 12\right)^{1-\alpha}\right)^2}{2^{3-2\alpha}},$$ and holds when $z$ tends to 1. In addition, as $\alpha \to 1^-$, this quantity tends to $1/2$.

\subsection{Convex mappings of order $\alpha$}\label{convex order alpha}

The class $C_\alpha$ consists of locally univalent normalized analytic functions $f$ on $\D$ satisfying
$$\operatorname{Re}\left\{1+z\dfrac{f''}{f'}(z)\right\}\geq \alpha, \quad 0\leq \alpha <1.$$
Since $CC_\alpha \subsetneq C_\alpha$, the bounds here will be weaker than those of the previous section, reflecting the coarser nature of the one-point condition.

\begin{prop}\label{corbeta}
 Let $f\in C_\alpha$, $0\leq \alpha<1$, then $\operatorname{Re}\left\{\dfrac{f(z)}{z}\right\}\geq\left(\dfrac{1}{4}\right)^{1-\beta}$ for all $z\in \D$, where
\begin{equation}\label{beta}
\beta=\frac{2\alpha-1+\sqrt{(2\alpha-1)^2+8}}{4}.  
\end{equation}
\end{prop}

\begin{proof}
Since $f\in C_\alpha$, by Theorem 1 in \cite{J1971}, it follows that $f\in S^*_\beta$. Thus, the statement follows  as in the proof of i) in Theorem \ref{CCC}.
\end{proof}

\begin{theo}
 Let $f\in C_\alpha$, $0\leq \alpha<1$, then
 \begin{equation*}
 \operatorname{Re}\{a_2f(z)\}\geq -\frac{1}{2}+ \left(\frac{1}{4}\right)^{1-\beta}\left[\delta_x\left(\frac{1}{4}\right)^{1-\beta}|z|(1-|z|^2)+\frac{1}{2}(1-|z|)^2\right],
 \end{equation*}
 for all $x,\,z\in \D$, where $\beta$ is given by (\ref{beta}) and
 \begin{equation}\label{order1}
\delta_x=\frac{2\delta-1+\sqrt{(2\delta-1)^2+8}}{4}, 
\end{equation}
with $\delta=\alpha\dfrac{1-|x|}{1+|x|}.$
\end{theo}

\begin{proof}
Since $f\in C_{\alpha}$,
\begin{equation*}
f_x(z):=\frac{f(\frac{z+x}{1+\overline{x}z})-f(x)}{(1-|x|^2)f'(x)}\in C_{\delta},
\end{equation*}
for all $x\in \D$, where $\delta:=\dfrac{1-|x|}{1+|x|}$ (see \cite{CH2023}). Thus, from Theorem 1 in \cite{J1971}, it follows that $f\in S^*_{\delta_x}$ with $\delta_x$ given by (\ref{order1}). Therefore
\begin{equation*}
    \operatorname{Re}\left\{z\frac{f'_x}{f_x}(z)\right\}\geq \delta_x,
\end{equation*}
for all $z\in \D$. 
Let
\[
G(z) := \frac{z\frac{f'_x}{f_x}(z)-\delta_x}{1-\delta_x}=\frac{1}{1-\delta_x}\left[\frac{zf'(\frac{z+x}{1+\overline{x}z})\frac{1-|x|^2}{(1+\overline{x}z)^2}-\delta_x}{f(\frac{z+x}{1+\overline{x}z})-f(x)}\right].
\]
Then $G$ is analytic in $\mathbb{D}$, satisfies $G(0)=1$ and $\operatorname{Re}\{G(z)\} \geq 0$, for all $z \in \mathbb{D}$. Thus, there exists an analytic function
$w : \mathbb{D} \to \mathbb{D}$, with $w(0)=0$, such that
\[
G(z)=\frac{1+w(z)}{1-w(z)}.
\]
Observe that
\[
w(-x)=\frac{G(-x)-1}{G(-x)+1}\qquad
 \text{and} \qquad w'(z)=\frac{2G'(z)}{(G(z)+1)^2},\]
whence Schwarz-Pick’s lemma implies
\begin{equation}\label{SP}
|w'(-x)| \leq \frac{1-|w(-x)|^2}{1-|x|^2}.
\end{equation}
Substituting the above expressions into (\ref{SP}), we obtain
\begin{equation}\label{ineq1}
(1-|x|^2)|G'(-x)| \leq 2\operatorname{Re}\{G(-x)\}.
\end{equation}
On the other hand,
\begin{equation}\label{ineq2}
G(-x)=\frac{1}{1-\delta_x}\left[\frac{x}{f(x)(1-|x|^2)}-\delta_x\right], 
\end{equation}
and a direct computation shows that
\begin{equation}\label{derG}
G'(-x)=-\frac{1}{(1-|x|^2)^2(1-\delta_x)}\frac{x}{f^2(x)}\left[(1+|x|^2)\frac{f(x)}{x}-(1+2a_2f(x))
\right].
\end{equation}
Now, we substitute \eqref{ineq2} and \eqref{derG} into \eqref{ineq1} to obtain
\begin{align*}
\left|(1+|x|^2)\frac{f(x)}{x}-(1+2a_2f(x))\right|&\leq
2\left[\operatorname{Re}\left\{ \frac{x\overline{f(x)}}{|x|}\right\}-\frac{\left|f(x)\right|^2}{|x|}\delta_x(1-|x|^2)\right]\\
&= 2|x|\left[\operatorname{Re}\left\{ \frac{f(x)}{x}\right\}-\left|\frac{f(x)}{x}\right|^2\delta_x(1-|x|^2)\right],
\end{align*}
whence
\[
\operatorname{Re}\{1+2a_2f(x)\}\geq
\operatorname{Re}\left\{\frac{f(x)}{x}\right\}(1-|x|)^2+2|x|\left|\frac{f(x)}{x}\right|^2\delta_x(1-|x|^2).
\]
Using the estimate in Proposition \ref{corbeta}, we finally obtain
\[
\operatorname{Re}\{1+2a_2f(x)\}\geq\left(\frac{1}{4}\right)^{1-\beta}\left[(1-|x|)^2+2|x|\left(\frac14\right)^{1-\beta}\delta_x(1-|x|^2)
\right].
\]
\end{proof}

\subsection{Spherically convex mappings}
\label{subsec:sconv}

Spherically convex functions were introduced by Minda \cite{Minda1986} as univalent mappings $f$ on $\D$ with $f(0)=0$, $f'(0)=\alpha\in(0,1]$, whose image $f(\D)$ is a spherically convex subset of the Riemann sphere, meaning convex with respect to spherical geodesics. We refer the reader to \cite{MP2000} for further properties of this class.

\begin{theo}
Let $f(z)=\alpha z + a_2 z^2 + \cdots$, $0<\alpha\leq 1$, be a spherically convex function. Then
\begin{equation*}
    \operatorname{Re}\{a_2f(z)\}\geq -\frac{\alpha^2}{2},
\end{equation*}
for all $z\in \D$. The constant $-\alpha^2/2$ is the best possible.
\end{theo}
\begin{proof}
    Let $\omega\in \overline{f(\D)}$. Since $f$ is s-convex and $f(0)=0$, we know (see \cite{MP2000}) that 
\begin{equation}\label{sphe1}
    \operatorname{Re}\left\{\frac{x+z}{x-z}+\frac{2zf'(z)}{f(z)-f(x)}
    \frac{1+\bar{\omega} f(x)}{1+\bar{\omega} f(z)}\right\}\geq0,    
    \end{equation}
    for all $x,z\in \D$. Fix $x\in \D$ and define
\[
A(z)=\frac{x+z}{x-z}+\frac{2zf'(z)}{f(z)-f(x)}    \frac{1+\bar{\omega} f(x)}{1+\bar{\omega} f(z)}.
\]
A straightforward calculation shows that
\[
A'(0)=\frac{2}{x} - \frac{2\alpha(1+\bar{\omega}f(x))}{f(x)} \quad\text{and} \quad A''(0)=\frac{4}{x^2}-\frac{4(1+\bar{\omega}f(x))}{f^2(x)}\left(\alpha^2(1-\bar{\omega}f(x))+ 2a_2 f(x)\right).
\]
On the other hand, by (\ref{sphe1}) there is $k:\D \to \D$, with $k(0)=0$, such that
\begin{equation*}
   A(z)=\frac{1+k(z)}{1-k(z)}.
\end{equation*}
It follows from the Schwarz-Pick lemma
\[
\frac{|A''(0)-A'(0)^2|}{4}
\leq 1-\left|\frac{A'(0)}{2}\right|^2,
\]
therefore
\begin{multline*}
\left|\frac{1}{x^2}-\frac{1+\bar{\omega}f(x)}{f(x)^2}(\alpha^2(1-\bar{\omega}f(x)) + 2a_2 f(x))-\left(\frac{1}{x}-\frac{\alpha(1+\bar{\omega}f(x))}{f(x)}\right)^2
\right|\leq \\ 1-\left|\frac{1}{x}-\frac{\alpha(1+\bar{\omega}f(x))}{f(x)}
\right|^2,
\end{multline*}
equivalently
\begin{multline*}
2|x|^2|1+\bar{\omega}f(x)|\left|\alpha \frac{f(x)}{x}-(a_2 f(x)+\alpha^2)\right|\leq \\ (|x|^2-1)|f(x)|^2+2Re\left\{\alpha(1+\bar{\omega}f(x))x\overline{f(x)}\right\}-\alpha^2|x|^2|1+\bar{\omega}f(x)|^2.
\end{multline*}
Using some algebraic manipulations, we obtain
\[
\operatorname{Re}\{a_2f(x)+\alpha^2\}\geq\alpha \operatorname{Re}\left\{\frac{f(x)}{x}\right\}-\alpha \operatorname{Re}\left\{ \frac{1+\omega\overline{f(x)}}{|1+\bar{\omega}f(x)|} \frac{f(x)}{x}\right\} +\frac{\alpha^2}{2}|1+\bar{\omega}f(x)|+\frac{|f(x)|^2(1-|x|^2)}{2|x|^2|1+\bar{\omega}f(x)|}.
\]
Now, with $\omega=f(x)$ in the above inequality, it follows that
\[
\operatorname{Re}\{a_2f(x)\}\geq-\frac{\alpha^2}{2}+\frac{\alpha^2}{2}|f(x)|^2+\frac{|f(x)|^2 (1-|x|^2)}{2|x|^2(1+|f(x)|^2)}.
\]
For the final statement, we consider the function 
$$f_{\alpha}(z)=\frac{\alpha z}{1-\beta z}=\alpha z+ \alpha\beta z^2+\cdots, \qquad \beta=\sqrt{1-\alpha^2},$$
which maps $\D$ conformally onto a hemisphere of the Riemann sphere; hemispheres are spherically convex by definition (see \cite{MP2000}, Proposition 2). Observe that 
\[
\operatorname{Re}\{a_2f_{\alpha}(z)\}=\alpha^2 \operatorname{Re}\left\{\frac{\beta z}{1-\beta z}\right\}
\]
and the infimum of $\operatorname{Re}\left\{\dfrac{\beta z}{1-\beta z}\right\}$ over $z\in \D$ is $-\dfrac{\beta}{1+\beta}\in(-1/2,0)$.
\end{proof}

\subsection{Uniformly starlike functions}
\label{subsec:UCV}

Goodman \cite{G1991, Goodman1991} introduced the class $\mathcal{UST}$ of uniformly starlike functions. A normalized function $f$ belongs to $\mathcal{UST}$ if it maps every circular arc $\gamma\subset\D$ with center $a\in \D$ onto a starlike arc with respect to $f(a)$. The analytic characterization is as follows
\begin{equation*}
    f\in\mathcal{UST} \iff \operatorname{Re}\left\{\frac{f'(z)(z-x)}{f(z)-f(x)}\right\}>0, \qquad x,z\in \D.
\end{equation*}
Note that the condition is again of two-point type, placing these classes naturally alongside $CC_\alpha$ in the context of this paper.

\begin{theo}
Let $f \in \mathcal{UST},$ then
\[
\operatorname{Re}\{a_2 f(z)\}\geq -\frac{1}{2}+\frac{1}{2}(1-2|z|)\operatorname{Re}\left\{\frac{f(z)}{z}\right\},
\]
for all $z\in\mathbb{D}$.
\end{theo}

\begin{proof}
Fix $x\in\mathbb{D}$ and define
\[
A(z)=\frac{f'(z)(z-x)}{f(z)-f(x)}, \qquad z\in \D.
\]
Since $f\in\mathcal{UST}$, we have $\operatorname{Re}\{A(z)\}>0$, which implies that
\begin{equation}\label{UST1}
  |A'(0)| \leq 2\operatorname{Re}\{A(0)\}.
\end{equation}
We observe that, by direct computation,
\[
A(0)=\frac{x}{f(x)}
\qquad \text{and}\qquad
A'(0)=\frac{x}{f^2(x)}-\frac{1-2a_2 x}{f(x)}.
\]
Thus, substituting into inequality \eqref{UST1}, we obtain
\[
\left|x-(1-2a_2x)f(x)\right|\leq 2 |f(x)|^2 \operatorname{Re}\left\{\frac{x}{f(x)}\right\},
\]
equivalently,
\[
\left|\frac{f(x)}{x}-1-2a_2f(x)\right|\leq 2|x| \operatorname{Re}\left\{\frac{f(x)}{x}\right\}.
\]
It follows that
\[
\operatorname{Re}\{2a_2 f(x)\}\geq -1+(1-2|x|) \operatorname{Re}\left\{\frac{f(x)}{x}\right\}.
\]
\end{proof}
Recall that in the previous theorem $\operatorname{Re}\{f(x)/x\}>0$, since $f\in S^*$. 
\subsection{Nehari classes}
\label{subsec:Nehari}

Schwarzian derivative of a locally univalent function $f$ is defined by
\[
Sf=\left(\frac{f''}{f'}\right)'-\frac{1}{2}\left(\frac{f''}{f'}\right)^2,
\]
a quantity that vanishes precisely for Möbius transformations and measures, in a precise sense, how far $f$ deviates from them. The Nehari class $\mathcal{N}$ consists of all locally univalent analytic functions on $\D$ with
\[
|Sf(z)|\leq \frac{2}{(1-|z|^2)^2},
\]
and Nehari's theorem states that every $f\in\mathcal{N}$ is univalent. More generally, the class $\mathcal{N}_t$, formally introduced in \cite{COP1996}, consists of locally univalent analytic functions satisfying
\[ 
|Sf(z)|\leq \frac{2t}{(1-|z|^2)^2}, \quad 0\leq t \leq 1;
\]
note that $\mathcal{N}_1=\mathcal{N}$, while $\mathcal{N}_0$ consists only of Möbius transformations.

\begin{theo}
   Let $f \in \mathcal{N}_t$, $0\leq t<1$, be normalized by 
$f(z)=z+a_2z^2+ \cdots$, $z \in \D$. Then 
\begin{equation*}
  \operatorname{Re}\{a_2f(z)\}\geq -\frac{1}{2} -\frac{t}{4(t+1)}\left[\left(\frac{1+|z|}{1-|z|}\right)^{\sqrt{1+t}}-1\right]^2,
\end{equation*}
for all $z\in \D$.
\end{theo}

\begin{proof}
We define 
\begin{equation*}
  g(z)=\frac{f(z)}{1+a_2f(z)}, \quad z\in \D.
\end{equation*}  
It is easy to show that $g(0)=0$, $g'(0)=1$, $g''(0)=0$,  $g \in \mathcal{N}_t$, $0\leq t<1$. By Theorem\,1 in \cite{CO1993}
\[\frac{1}{\sqrt{1+t}}\leq |g(z)| \leq \frac{1}{\sqrt{1-t}},\]
for all $z\in \D$. Thus, 
\begin{equation}\label{N1}
  1-t \leq \frac{|1+a_2f(z)|^2}{|f(z)|^2}=\frac{1+|a_2|^2|f(z)|^2+2\operatorname{Re}\{a_2f(z)\}}{|f(z)|^2}.
\end{equation}
On the other hand, Ch. Pommerenke in \cite{P1964} proved that
\begin{equation}\label{N2}
  |a_2|\leq \sqrt{1+t} \quad \text{and} \quad |f(z)|\leq \frac{1}{2\sqrt{t+1}}\left[\left(\frac{1+|z|}{1-|z|}\right)^{\sqrt{1+t}}-1\right],
\end{equation}
for all $z\in \D$. Therefore, by \eqref{N1} and \eqref{N2}
\begin{equation*}
    2\operatorname{Re}\{a_2f(z)\}\geq -1-2t|f(z)|^2\geq -1-\frac{t}{2(t+1)}\left[\left(\frac{1+|z|}{1-|z|}\right)^{\sqrt{1+t}}-1\right]^2.
\end{equation*}
\end{proof}

\begin{rem}
As $t\to 0^+$ the lower bound tends to $-1/2$, consistently with the fact that $\mathcal{N}_0$ reduces to the class of Möbius transformations for which $\sigma(\mathcal{N}_0)=\{w:\operatorname{Re}\{w\}>-1/2\}$, matching the classical result for convex mappings.
\end{rem}

\end{document}